\documentclass[a4paper]{article}

\usepackage{
amsmath,
amsthm,
amscd,
amssymb,
}
\usepackage{comment}
\usepackage{tikz}
\usetikzlibrary{positioning,arrows,calc}
\usepackage{xspace}
\usepackage{wasysym}

\usepackage{xint}
\def\allowsplits #1{\ifx #1\relax \else #1\hskip 0pt plus 1pt\relax
\expandafter\allowsplits\fi}%
\def\printnumber #1{\expandafter\allowsplits \romannumeral-`0#1\relax }%

\usepackage{graphicx}

\usepackage{url}

\usepackage{listings}

\theoremstyle{plain}
\newtheorem{lemma}{Lemma}
\newtheorem{definition}{Definition}
\newtheorem{corollary}{Corollary}

\newtheorem{theorem}{Theorem}

\newtheorem{remark}{Remark}

\newtheoremstyle{derp}
{3pt}
{3pt}
{}
{}
{\upshape}
{:}
{.5em}
{}
\theoremstyle{derp}
\newtheorem{example}{Example}

\newcommand{\Z}{\mathbb{Z}}

\newcommand{\N}{\mathbb{N}}

\newcommand{\ID}{\mathrm{id}}

\newcommand\xqed[1]{%
  \leavevmode\unskip\penalty9999 \hbox{}\nobreak\hfill
  \quad\hbox{#1}}
\newcommand\qee{\xqed{$\fullmoon$}}

\newcommand{\Sym}{\mathrm{Sym}}
\newcommand{\Alt}{\mathrm{Alt}}

\newcommand{\OB}{\mathrm{OB}}
\newcommand{\RTM}{\mathrm{RTM}}

\title{Universal gates with wires in a row}

\author{
Ville Salo \\
vosalo@utu.fi
}

\begin{document}
\maketitle

\begin{abstract}
We give some optimal size generating sets for the group generated by shifts and local permutations on the binary full shift. We show that a single generator, namely the fully asynchronous application of the elementary cellular automaton 57 (or, by symmetry, ECA 99), suffices in addition to the shift. In the terminology of logical gates, we have a single reversible gate whose shifts generate all (finitary) reversible gates on infinitely many binary-valued wires that lie in a row and cannot (a priori) be rearranged. We classify pairs of words $u, v$ such that the gate swapping these two words, together with the shift and the bit flip, generates all local permutations. As a corollary, we obtain analogous results in the case where the wires are arranged on a cycle, confirming a conjecture of Macauley-McCammond-Mortveit and Vielhaber.
\end{abstract}

\section{Introduction}

In this paper, we study the subgroup $G$ of the self-homeomorphism group of $\{0,1\}^\Z$ generated by shifts and local permutations (or reversible finitary logical gates) on the set of all bi-infinite binary sequences, studied previously in \cite{BaKaSa16} under the name $\OB(2,1)$ (as a subgroup of the larger group $\RTM(2, 1)$). 

We prove a result about generators for the group $G$: In Section~\ref{sec:GeneratingByWS} we characterize those pairs of words $u, v$ such that the homeomorphism swapping $u$ and $v$ at the origin, together with the shift and the bit flip, generates all of $G$. In the process we identify some natural subgroups of $G$ in Section~\ref{sec:Subgroups}, with constraints related to linearity, and to (group-)geometric phenomena such as one-sided information flow.

We say a set of local permutations is a \emph{universal gate set} if it, together with the shift, generates $G$. As an application of the result about word swaps, we prove that ECA 57 is a singleton universal gate set by showing that it generates both a universal word swap and the bit flip. In particular, it follows that $G$ is $2$-generated.

In \cite{MaMcMo11,Vi12}, it is asked whether asynchronous ECA $57$ is universal in the case when wires are arranged on a cycle $\Z/n\Z$. We solve this question by showing that universality on $\Z$ implies universality on $\Z/n\Z$ for all large enough $n$ in Section~\ref{sec:ftofn}. Impatient readers can find copy-pasteable implementations\footnote{The formulas are obtained by direct application of the straight-line grammar deduced from the proof, and are probably not even close to optimal. Eight gates can be erased from the Toffoli gate implementation by simply erasing repeated generators.} of standard gates in terms of shifts of asynchronous ECA $57$ in the appendices.

We find the group $G$ interesting for several reasons: From the perspective of computation and logical gates, $G$ is precisely the computation you can do with reversible gates if your infinitely many wires are arranged in a row. It differs from the groupoid \cite{La03} and clone formalisms \cite{Bo15} most concretely in the aspect that the model has infinitely many borrowed bits (in the sense of \cite{Xu15}) which arise naturally from the infiniteness of $\Z$. Reordering of the wires is not built into the formalism, which strengthens the notion of universality. 

The group $G$ seems like a natural simplest group that can be used for the geometric group theory of reversible logic gates. 

The author finds $G$ and its variants useful in the context of symbolic dynamics, and has applied $G$ and its local embeddings (in the LEF sense \cite{CeCo10}) in several works on automorphism groups, see e.g. \cite{Sa17b,Sa18a}. 


\section{Definitions}

We assume $0 \in \N$ and in $f \circ g = fg$, $g$ is applied first. All intervals $[a,b]$, as well as half-open and open ones, are interpreted inside $\Z$. For two words $u, v \in \{0, 1\}^n$ write $D(u, v) \subset \{0,...,n-1\}$ for the set of coordinates where $u$ and $v$ differ, and $d(u, v) \in \{0, 1\}^n$ for the characteristic function of $D(u, v)$. Words are $0$-indexed by default.

We consider $X = \{0,1\}^\Z$ with the product topology arising from the discrete topology of $\{0, 1\}$. \emph{The shift} is $\sigma : X \to X$ defined by $\sigma(x)_i = x_{i + 1}$. A power $\sigma^n$ is also informally referred to as \emph{a shift}, which should not cause confusion. For $n \geq 1$, the \emph{$n$-periodic points} are the points $x \in X$ satisfying $\sigma^n(x) = x$. The set of $n$-periodic points is denoted by $X_n$. The \emph{shifts} of a function are its conjugates by shifts $\sigma^n$.

There is an obvious way to see $X$ as a vector space over the two-element field, and addition and terms such as \emph{linear} and \emph{affine} for maps on $X$ are interpreted with respect to this structure. The subset $X_0 = \{x \in X \;|\; \exists R: \forall |i| \geq R: x_i = 0\}$ is a subspace. An \emph{affine translation} refers to a function $t : X \to X$ of the form $x \mapsto x + w$, with $w \in X_0$.

We use the $\langle F_1, ..., F_n \rangle$ notation to denote the smallest subgroup of a group $G$ (deduced from the context) containing the elements equal to $F_i$ (when $F_i \in G$) or found in $F_i$ (when $F_i$ is a set). 

\section{$G$}

\begin{definition}
The group $G$ is defined as the subgroup of the self-homeomorphism group of $X$ generated by the shift $\sigma$ and homeomorphisms $f : X \to X$ satisfying
\[ \exists r: |i| \geq r \implies f(x)_i = x_i. \]
\end{definition}

We make some basic observations about the structure of this group $G$ (see also \cite{BaKaSa16,Sa17b}).

We call self-homeomorphisms $f$ of $X$ satisfying
$\exists r: |i| \geq r \implies f(x)_i = x_i$
\emph{local permutations}, \emph{inert\footnote{This word is by analogue with subshift automorphism group terminology \cite{KiRo91}.} elements of $G$} or \emph{gates}. A set of gates $F$ is \emph{universal} if $G = \langle F, \sigma \rangle$.

For an inert element $f \in G$, the finite set of coordinates that may change under the application of $f$ is called its \emph{support}. Of more importance than the support and the number $r$ is the $R$ in the following lemma, which states that if $f$ is inert in $G$, then not only does it change only finitely many symbols in a given point of $X$, but it also \emph{looks at} only finitely many coordinates.

\begin{lemma}
\label{lem:LocalRule}
An element $f \in G$ is inert if and only if there exists a \emph{strong radius} $R$ and a \emph{local rule} $\hat f : \{0,1\}^{[-R,R]} \to \{0,1\}^{[-R,R]}$ such that $f(x)_i = x_i$ for all $|i| > R$, and $f(x)_i = \hat f(x_{[-R, R]})_i$ for $i \in [-R, R]$.
\end{lemma}

\begin{proof}
Clearly any $R$ and $\hat f$ define an inert element of $G$.

Suppose then $f \in G$ is inert, and let $r$ be as in the definition. Since $f$ is continuous, there exists $R$ such that $f(x)_{[-r, r]}$ is determined by $x_{[-R, R]}$, where we may assume $R \geq r$. Since $f(x)_{i} = x_i$ when $|i| \in [-R, -r-1] \cup [r+1, R]$, in fact $f(x)_{[-R,R]}$ is also determined by $x_{[-R,R]}$. So we can pick $\hat f(w) = f(x^w)_{[-R,R]}$ where $x^w \in \{0,1\}^\Z$ is any point with $x^w_{[-R,R]} = w$. Since $f$ is a homeomorphism, it is in particular a bijection, so $\hat f$ is also a bijection.
\end{proof}

Define the \emph{strong support} of inert $f$ as the set of coordinates that can be changed by $f$, and whose change may affect another coordinate, which is finite by the previous lemma.

The number $R$ can be arbitrarily larger than $r$, even if $r = 0$.

\begin{remark}
We note that (by easy counterexamples) the inert elements of $G$ are not the same thing as self-homeomorphisms of $X$ which have clopen support in the usual sense of homeomorphisms (where the support of a homeomorphism is the closure of the complement of its set of fixed points). Neither are they the same thing as self-homeomorphisms of $X$ which always modify only finitely many coordinates, or ones that always modify only a bounded number of coordinates. By the discussion above, they \emph{are} precisely the self-homeomorphisms of $X$ which always modify a bounded \emph{set} of coordinates. 
\end{remark}

\begin{lemma}
Elements of $G$ can be written in a unique way as $\sigma^n \circ f$ with $f$ inert. There is a unique homomorphism $\alpha : G \to \Z$ satisfying $\alpha(\sigma) = 1$. The kernel of $\alpha$ consists of the inert elements. The group $G$ is (locally finite)-by-$\Z$, more precisely a semidirect product of $\Z$ acting on the locally finite group of inert elements.
\end{lemma}

The inert map $g$ in $f = \sigma^n \circ g$ is called the \emph{inert part} of $f$.

\begin{proof}
Since the local permutations form a group themselves, and this group is normalized by $\sigma$ (i.e. conjugation of an inert element by $\sigma$ gives an inert element), it follows that every element of $G$ is of the form $\sigma^n \circ f$ where $f$ is inert, and clearly $n$ is uniquely determined, after which $f$ is clearly uniquely determined. Any homomorphism to $\Z$ maps any inert $f$ to $0$ (since inert elements are of finite order), so by the normal form, $\sigma \mapsto 1$ uniquely determines the homomorphism.

It follows from this that $\ker \alpha$ consists of precisely the inert elements, and inert elements clearly form a locally finite group, since the support of the composition of two inert element is at most the union of the support of the generators. This gives $G$ the structure of a (locally finite)-by-$\Z$ group, and since the homomorphism $\alpha$ is split, $G$ is in fact a semidirect product of $\Z$ acting on the locally finite group of inert elements.
\end{proof}

We end this section with some general discussion.

The inert elements can be naturally considered reversible gates by Lemma~\ref{lem:LocalRule}, as they perform an operation on finitely many wires of the input independently of the other wires. From this perspective, the shift should be seen as either a technical tool or as an \emph{internal symmetry} of the group -- shifts are not necessarily computationally meaningful, but \emph{conjugation by a shift} is very meaningful, as it simply means applying a gate to another contiguous set of wires, and this is precisely what it means to say that $G$ is a semidirect product of $\Z$ acting on the inert elements.

Though we do not do so here, from this point of view of logical gates it would be natural to see $G$ not just as a group but as a group where the shifts are part of the structure. A group $H$ with a fixed epimorphism $\pi : H \to \Z$ is often called \emph{($\Z$-)indicable}, so one could sensibly call a group with a fixed split epimorphism $\pi : H \to \Z$ \emph{strongly $\Z$-indicable}, with the understanding that $\pi$ and the right inverse can be taken to be part of the structure, if needed. In particular if the $\Z$-embedding giving the right inverse is taken part of the structure, then substructures of $G$ always contain the shifts.

We mention some standard notions in category theory with are related: One can see a group $H$ with a preferred $\Z$-subgroup as a pair $(H, \psi)$ where $\psi : \Z \to H$ is a homomorphism (that is part of the structure). Such pairs form the \emph{under category}, also known as the \emph{coslice category}, $\Z \downarrow \mathrm{Grp}$. This does not enforce the semidirect product structure in the sense that the $\Z$ picked out by $\psi$ can behave very badly in other objects of this category, but this does enforce the property that subobjects of $G$ automatically contain the shifts (note that monomorphisms in Grp are the injections). Dually, one can see $G$ as an indicable group with a fixed $\pi : G \to \Z$, i.e. an element of the slice category $\mathrm{Grp} \downarrow \Z$, but this does not force subobjects of $G$ to contain the shifts. 



\section{$G_n$}
\label{sec:ftofn}

\begin{definition}
For any $n \in \N$, write $G_n = \Alt(\{0,1\}^n)$.
\end{definition}

The notation is meant to imply that $G_n$ can be seen as a modulo $n$ variant of $G$. In this section, we develop an explicit algebraic connection.

For inert $f$, the \emph{strong shift-invariant radius} of $f$ is the minimal $R$ among strong radii of conjugates $\sigma^{m} \circ f \circ \sigma^{-m}$, and for general $f \in G$ the strong shift-invariant radius is the strong shift-invariant radius of the inert part. The importance of the strong shift-invariant radius is the following.

\begin{definition}
\label{def:ftofn}
Let $f = \sigma^m \circ g \in G$, with $g$ inert with strong shift-invariant radius $R$. 
Suppose $n > 2R+1$. Then for $x \in X_n$ define
\[ \bar f_n(x) = \sigma^m \circ \lim_{k \rightarrow \infty} \left(\prod_{i = -k}^k \sigma^{-in} \circ g \circ \sigma^{in}\right)(x) \]
For $w \in \{0,1\}^n$ define $f_n(w) = \bar f_n(...www.www...)_{[0,n-1]}$.
\end{definition}

In particular, $\sigma_n(aw) = wa$ for $a \in \{0,1\}, w \in \{0,1\}^{n-1}$.

Since $n > 2R+1$, the limit is well-defined as the maps $\sigma^{-in} \circ g \circ \sigma^{in}$ commute pairwise (for distinct values of $i$) and every $j \in \Z$ is in the support of only finitely many of them (in fact only one of them). Observe that $\bar f_n$ commutes with $\sigma^n$, and thus preserves the set of $n$-periodic points. The action of $\bar f_n$ on the $n$-periodic points of $\{0,1\}^\Z$ fully determines the action of $f_n$ on $\{0,1\}^n$, and vice versa.

\begin{remark}
\label{rem:Formula}
For completeness, we give a formula for $f_n$. Let $f \in G \setminus \{\ID\}$ be inert, let the strong shift-invariant radius of $f$ be $R$, and let $m$ be minimal such that $\sigma^m \circ f \circ \sigma^{-m}$ has strong radius $R$, and let $\hat f : \{0,1\}^{[-R, R]} \to \{0,1\}^{[-R, R]}$ be the local rule for $\sigma^m \circ f \circ \sigma^{-m}$. If $[m-R, m+R] \equiv [a, b] \bmod n$ for some $a, b \in [0,n-1]$ then define
\[ f_n(w) = w_{[0,a)} \cdot \hat f(w_{[a, b]}) \cdot w_{(b, n-1]}, \]
and if $[m-R, m+R] \equiv [0,a) \cup (n-1-b,n-1] \bmod n$ for some $a, b \in [0,n-1]$ (where possibly $a = 0$ or $b = 0$) then define
\[ f_n(w) = \hat f(u)_{[b,2R+1)} \cdot w_{[a, n-1-b]} \cdot \hat f(u)_{[0, b)} \]
where $u = w_{(n-1-b, n-1]} \cdot w_{[0, a)}$. For $f = \sigma^k \circ g$ with $g$ inert, we have $f_n = \sigma_n^k \circ g_n$
\end{remark}

Note that if $m, R, \hat f$ are as above, then $f$ applies its local rule $\hat f$ ``at $m$'', i.e. in the coordinates $[m-R, m+R]$. 

\begin{lemma}
\label{lem:Sufficientn}
Let $f = \sigma^k \circ g$ with $g$ inert with strong shift-invariant radius $R$. If $n > 2R+1$, then $f_n \in G_n$. 
\end{lemma}

\begin{proof}
Consider first inert $g$ with strong shift-invariant radius $R$. Since $n > 2R+1$, by the formulas of Remark~\ref{rem:Formula}, $g_n$ permutes only the contents of some $2R + 1$ coordinates of $[0,n-1]$, so in fact $g_n$ performs $2^{n-(2R+1)}$ ``copies'' of the same permutation depending on the values of the other $n - 2R - 1$ coordinates, thus it is a composition of an even number of permutations with disjoint supports, all having the same parity, thus $g_n \in G_n$.

The shift $\sigma_n$ is a permutation of the $n$ coordinates of a given word, so it can be written as a product of swaps of adjacent coordinates. When $n \geq 3$, coordinate swaps are even by the argument of the above paragraph, and thus $\sigma_n$ is even as a composition of even permutations. Thus for $f = \sigma^k \circ g$ with $g$ inert, we have $f_n = \sigma^k_n \circ g_n \in G_n$.
\end{proof}

The proof shows that $\sigma_n \in G_n$ whenever $n \geq 3$. We mention a combinatorial consequence of this: The parity of a permutation on a set $X$ is the parity of $|X| - c$ where $c$ is the number of cycles in the cycle decomposition. Thus the sign of $\sigma_n$ is $2^n - p_n$ where $p_n$ is the number of orbits of words of length $n$ under the cyclic shift. The sequence $p_n$ is well-known, it is sequence A000031 in the OEIS, and a standard formula for it is $p_n = \frac{1}{n} \sum_{d|n} \phi(d) 2^{n/d}$ (where $\phi$ is the Euler totient function). We are stating precisely that $p_n$ is even for all $n \geq 3$. Since $p_1 = 2$ and $p_2 = 3$, $p_2$ is the only odd value in the sequence $p_1, p_2, p_3, ...$.


\begin{lemma}
\label{lem:ConjugationThingie}
Let $g$ be inert with strong shift-invariant radius $R$, and suppose $n > 2R + 1$. Then $g_n \circ \sigma_n^m = \sigma_n^m \circ (\sigma^{-m} \circ g \circ \sigma^m)_n$ for all $m \in \Z$.
\end{lemma}

\begin{proof}
We prove instead the equivalent equality $\bar g_n \circ \bar \sigma_n^m = \bar \sigma_n^m \circ \overline{(\sigma^{-m} \circ g \circ \sigma^m)}_n$ for the actions on $n$-periodic points. Letting $x \in X_n$, we have
\begin{align*}
\bar g_n \circ \bar \sigma_n^m(x) &= \lim_{k \rightarrow \infty} \left(\prod_{i = -k}^k \sigma^{-in} \circ g \circ \sigma^{in}\right)(\sigma^m(x)) \\
&= \sigma^m \circ \lim_{k \rightarrow \infty} \left(\prod_{i = -k}^k \sigma^{-in} \circ (\sigma^{-m} \circ g \circ \sigma^m) \circ \sigma^{in}\right)(x) \\
&= \bar \sigma_n^m \circ \overline{(\sigma^{-m} \circ g \circ \sigma^m)}_n(x)
\end{align*}
where $\sigma^m$ waddles over the finite expressions by using abelianity of $\Z$ and by conjugating $g$s, and the result stays correct in the limit since $\sigma$ is continuous. 
\end{proof}

The map $f \mapsto f_n$ does not give a homomorphism from $G$ to $G_n$ (it is not even well-defined for all elements of $G$), and even if we induce a mapping from $G$ to $G_n$ using a finite generating set of $G$, it will typically not give a homomorphism (the group $G$ is not residually finite).

However, this mapping does give us local homomorphisms in the sense of LEF groups \cite{CeCo10}: as long as information is ``not passed around the cycle'', we see the same relations in both $G$ and $G_n$. The following lemma is a weak version of this fact:

\begin{lemma}
\label{lem:WeakLEF}
Let $f^1, ..., f^k \in G \setminus \{\ID\}$ with $f^i = \sigma^{\ell_i} \circ g^i$, where $g^i$ is inert. Let $g^i$ have strong shift-invariant radius $R_i$ and let $m_i$ be minimal such that $\sigma^{m_i} \circ g_i \circ \sigma^{-m_i}$ has strong radius $R_i$. Let $t = \sum_j |\ell_j|$. Suppose that for some $h$, $[m_i-t-R_i, m_i+t+R_i] \subset [h, h+n-1]$ for all $i$. 
Then $(f^1 \circ \cdots \circ f^k)_n = f^1_n \circ \cdots \circ f^k_n$.
\end{lemma}

In particular this lemma shows that $G$ is indeed a LEF group.

\begin{proof}
First, we reduce to the inert case.

Observe that by iterating Lemma~\ref{lem:ConjugationThingie} the RHS of the equality is
\[ f^1_n \circ \cdots \circ f^k_n = \sigma^{\ell_1}_n \circ g^1_n \circ \cdots \circ \sigma^{\ell_k}_n \circ g^k_n = \sigma^{\sum_j \ell_j}_n \circ h^1_n \circ \cdots \circ h^k_n \]
where $h^i = \sigma^{-t_i} \circ g^i \circ \sigma^{t_i}$ for some $|t_i| \leq t$ (where $t_i$ is some partial sum). In particular, the $h^i$ are inert with strong shift-invariant radius $R_i$. The LHS on the other hand is
\[ (f^1 \circ \cdots \circ f^k)_n = (\sigma^{\sum_j \ell_j} \circ h^1 \circ \cdots \circ h^k)_n = \sigma^{\sum_j \ell_j}_n \circ (h^1 \circ \cdots \circ h^k)_n \]
by definition.

We may thus assume $\ell_i = 0$ for all $i$, by replacing each $m_i$ with $m_i + t_i$ (after which the corresponding assumption $[m_i - R_i, m_i + R_i] \subset [h, h+n-1]$ still holds), and by cancelling the common part $\sigma^{\sum_j \ell_j}_n$. We have then reduced the problem to the case where the $f^i = g^i = h^i$ are inert.

Let us now prove the equivalent equality
\[ \overline{(f^1 \circ \cdots \circ f^k)}_n = \bar f^1_n \circ \cdots \circ \bar f^k_m. \]
on $X_n$. For this, observe that the action of the maps $f^i$ on the interval $[h, h+n-1]$ is precisely the same as the action of the maps $\hat f^i_n$, and the action of the $\hat f^i_n$ on any contiguous sequence of $n$ integers determines the action. Thus the map $f^i \mapsto \bar f^i_n$ even induces a homomorphism from $\langle f^1,...,f^k \rangle$ to $\langle \bar f^1_n, ..., \bar f^k_n \rangle$, in particular we have $(f^1 \circ \cdots \circ f^k)_n = f^1_n \circ \cdots \circ f^k_n$ as desired.
\end{proof}

\section{How not to generate $G$ (subgroups)}
\label{sec:Subgroups}

We now define some important self-homeomorphisms of $X$ (not all in $G$) and make simple observations about subgroups they generate. For the purpose of universality, the importance of understanding subgroups is of course that if a set $F \subset G$ is contained in a proper subgroup $F \subset H < G$, then also $\langle F \rangle \leq H < G$, preventing $F$ from generating $G$.

The \emph{reversal} $R : X \to X$ defined by $R(x)_i = x_{-i}$ is not an element of $G$, but normalizes $G$ inside the self-homeomorphism group of $X$ (so $G$ has index $2$ in $\langle R, G \rangle$). The \emph{swap} is $s \in G$ defined by $s(x)_0 = x_1, s(x)_1 = x_0, \forall i \notin \{0,1\}: s(x)_i = x_i$. The \emph{$k$-CNOT} is $c^k \in G$ defined by $c^k(x)_0 = 1-x_0 \iff x_{[1,k]} = 1^k$ and $\forall i \neq 0: c^k(x)_i = i$. The map $c^0$ inverts the symbol at the origin, and is usually called the \emph{NOT} gate in the theory of logical gates. We also refer to it as the \emph{flip}. The map $c^1$ is linear, and is called the \emph{CNOT} or \emph{controlled NOT}. It is useful to keep in mind the identity
\begin{equation}
\label{eq:cliche}
c^1 \circ (\sigma^{-1} \circ (R \circ c^1 \circ R) \circ \sigma) \circ c^1 = s
\end{equation}
which states the swap in terms of CNOTs. The gate $c^2$ is called the \emph{Toffoli gate}.

We make some observations about subgroups generated by these elements. Write $\Sym_{<\infty}(\Z)$ for the group of finite-support permutations on $\Z$.

\begin{lemma}
\label{lem:Subgroups}
\[ \langle s, \sigma \rangle = \{ f \in G \;|\; \exists \pi \in \Sym_{<\infty}(\Z), n \in \Z: \forall i, x: f(x)_i = x_{\pi(i) + n} \} \]
\[ \langle c^0, \sigma \rangle = \{ f \in G \;|\; \exists c \in X_0, n \in \Z: \forall x: f(x) = \sigma^n(x) + c \} \cong \Z_2 \wr \Z \]
\[ \langle c^1, Rc^1R, \sigma \rangle = \{ f \in G \;|\; \forall x, y: f(x + y) = f(x) + f(y) \} \]
\[ \langle c^0, c^1, Rc^1R, \sigma \rangle = \{ f \in G \;|\; \exists c \in X_0: \forall x, y: f(x + y) = f(x) + f(y) + c \} \]
\end{lemma}

\begin{proof}[Proof sketch]
In each case, it is straightforward to show that the generators are contained in the RHS, and that the RHS is a group. To see that all of the group is generated, for the first equality we need that every permutation can be decomposed into swaps. The second is clear, and this is a standard realization of the lamplighter group $\Z_2 \wr \Z$. The third follows from \eqref{eq:cliche} and the fact elementary matrices generate matrix groups. 
The RHS of the fourth equality is equivalent to $f$ being the sum of a linear map and a constant in $X_0$, and the subgroup $\langle c^0, \sigma \rangle$ allows addition of arbitrary constants.
\end{proof}

The first group is a standard example of a finitely-generated LEF group that is not residually finite \cite{CeCo10} ($G$ itself is also such an example).

The following subgroups are important in the classification of universal word swap gates, and we include a more complete treatment.

\begin{definition}
If $V \leq X_0$ is a shift-invariant vector space, let
\[ G_V = \{ f \in G \;|\; \exists n \in \Z, v \in X_0: \forall x: f(x) \in \sigma^n(x) + v + V \}. \]
Define
\[ G_R = \{ \sigma^n \circ f \;|\; f \in \ker(\alpha), n \in \Z, \exists (g_i)_{i \in \Z}: \forall x \in X: f(x)_i = g_i(x_{(-\infty, i]}) \]
where the $g_i$ are functions whose type is unambiguous from context. Define $G_L = R \circ G_R \circ R$
\end{definition}

Informally, $G_V$ is the subgroup of $G$ for which the cosets of $V$ are a system of blocks of imprimitivity, and these cosets are permuted by an affine translation.
The interpretation of $G_R$ is that information can only travel to the right, though literally this only holds in the inert part.

\begin{example}
For all $k$ we have
$c^k \in G_L$ and $Rc^kR \in G_R$.
We have $c^0 \in G_V$ for any $V$, and if $\pi$ is the permutation $(00 \; 11)(01)(10)$ of $\{0,1\}^2$, in cycle notation, then the map $f(x) = x_{(-\infty,0)}.\pi(x_{[0,1]}) x_{[2, \infty)}$ is in $G_V$ where $V = \{x \in X_0 \;|\; \sum_i x_i \equiv 0 \bmod 2\}$.

This map $f \in G_v$ (which is $f_{00,11}$ in the notation of Section~\ref{sec:GeneratingByWS}), is not linear, but is affine. The map that performs the permutation $(000 \; 111)$ is in $G_{V'}$ for a proper subspace $V' < X_0$, but is not even affine. \qee
\end{example}

\begin{lemma}
Let $V < X_0$ be a shift-invariant proper subspace. Then $G_V, G_L, G_R$ are proper subgroups of $G$. 
\end{lemma}

\begin{proof}
First, we show $G_V$ is a subgroup. Let us call the vector $v$ in the definition the \emph{offset vector} and $n$ the \emph{shift}. Let $f \in G_V$ with offset vector $v$ and shift $n$. Then for any $x$, $f(x + V) \subset \sigma^n(x) + \sigma^n(V) + v + V = \sigma^n(x) + v + V$ by shift-invariance of $V$, thus $f(x + V) = \sigma^n(x) + v + V$ since $f$ is a homeomorphism and cosets of $V$ partition $X$, thus
\[ f(\sigma^{-n}(x + v) + V) = \sigma^n(\sigma^{-n}(x + v)) + v + V = x + V \]
for any $x$, thus $f^{-1}(x + V) = \sigma^{-n}(x) + \sigma^{-n}(v) + V$, thus $f^{-1}(x) \in \sigma^{-n}(x) + \sigma^{-n}(v) + V$ for all $x$, so $f^{-1} \in G_V$ with offset vector $\sigma^{-n}(v)$ and shift $-n$.

Suppose then $f, g \in G_V$, and let the offset vectors be $u, v$, and shifts $m, n$, respectively. Then $f \circ g(x) \in f(\sigma^n(x) + v + V) \in \sigma^{m+n}(x) + \sigma^m(v) + u + V$, so $f \circ g$ admits offset vector $\sigma^m(v) + u$ and shift $m+n$, thus $f \circ g \in G_V$.

To see that $G_V$ is a proper subgroup, observe that since $V$ is shift-invariant and $V \neq X_0$, $V$ does not contain the vector ${^\omega}0.10^{\omega}$. Consider the map $f(x.aby) = x.\pi(ab)y$ where $\pi : \{0,1\}^2 \to \{0,1\}^2$ is the permutation $\pi = (00)(10)(01 \; 11)$ in cycle notation. Since $f(0^\Z) = 0^\Z$ and $f$ is inert, if $f \in G_V$ then necessarily the offset vector $v$ in the definition would satisfy $v \in V$. But then $f({^\omega}0.010^\omega) = {^\omega}0.110^\omega \in {^\omega}0.010^\omega + V$ implies ${^\omega}0.10^{\omega} \in V$, a contradiction.

Next, we show $G_R$ is a subgroup. Let us call the sequence $(g_i)_{i \in \Z}$ appearing in the definition the \emph{rule sequence}. Let $\sigma^n \circ f \in G_R$ with rule sequence $(g_i)_{i \in \Z}$. Then as above, it has inverse $\sigma^{-n} \circ (\sigma^n \circ f^{-1} \circ \sigma^{-n})$ and we only need to find a rule sequence for $\sigma^n \circ f^{-1} \circ \sigma^{-n}$.

First, we show that $f^{-1}$ has a rule sequence. Observe that necessarily $g_i(y) = y_i$ for all but finitely many $i$ (since $f$ is inert). Observe then that we have $g_i(ya) \neq g_i(yb)$ for all $y \in \{0,1\}^{(-\infty, i-1]}$, $a \neq b$. There are several ways to see this: it follows from the fact that $f^m = \ID$ for some $m$ (since $f$ is inert), or from a direct counting argument using the fact $g_i(y) = y_i$ for large $|i|$.

Now, we can build, by induction, a sequence of functions $h_i' : \{0,1\}^{(-\infty, i]} \to \{0,1\}^{(-\infty, i]}$ such that $h_i'(g_i(y)) = y_i$ for all $y \in \{0,1\}^{(-\infty, i]}$, and the functions agree in the sense $h_{i+1}'(ya)_{(-\infty, i]} = h_i'(y)$ for $y \in \{0,1\}^{(-\infty, i]}, a \in \{0,1\}$: Pick $h_i = \ID$ for all large enough $-i$, and
$h_{i+1}'(ya) = h_i'(y)b$
where $b$ is the unique letter such that $g_{i+1}(h_i'(y)b) = a$. Now, $h_i(y) = h_i'(y)_i$ defines a rule sequence for $f^{-1}$.

Now,
\begin{align*}
\sigma^n f^{-1} \sigma^{-n}(x)_i &= f^{-1} \sigma^{-n}(x)_{i + n} 
= h_{i+n}(\sigma^{-n}(x)_{(-\infty,i+n]}) = h_{i+n}(x_{(-\infty, i]}).
\end{align*}
where we shift the domain of $h_{i+n}$ to $h_{i+n} : \{0,1\}^{(-\infty, i]} \to \{0,1\}$ with the obvious interpretation. Thus, $(h_{i+n})_{i \in \Z}$ is a rule sequence for $\sigma^n \circ f^{-1} \circ \sigma^{-n}$.

Suppose then that $\sigma^n \circ f, \sigma^m \circ f' \in G_V$ with $f, f'$ inert. The above paragraph shows that the set of inert elements having a rule sequence is closed under conjugation by the shift, so by $\sigma^m f' \sigma^n f = \sigma^{m+n} (\sigma^{-n} f' \sigma^n) f$
we only need to show that if $f'$ and $f$ have rule sequences, then so does $f' \circ f$. This follows from
\[ f'f(x)_i = g'_i(f(x)_{(-\infty, i]}) = g'_i(y) \]
where $y_j = g_j(x_{(-\infty, j]})$ for $j \leq i$. This expression only looks at the values of $x$ in $x_{(-\infty, i]}$, thus defines a rule sequence.

To see that $G_R$ is a proper subgroup, observe that the swap $s \notin G_R$.

The proof for $G_L$ is symmetric.
\end{proof}

\section{How to generate $G$ and $G_n$}

The following is essentially folklore in the theory of reversible circuits.

\begin{lemma}
\label{lem:GeneratingGn}
Let $n \geq 4$. Letting
\[ S = \{ (\sigma^{-m} \circ s \circ \sigma^m)_n \;|\; m \in [0, n-2] \}, \]
we have $G_n = \langle S, c^0_n, c^2_n \rangle$. In particular,
\[ G_n = \langle s_n, c^0_n, c^2_n, \sigma_n \rangle. \]
\end{lemma}

\begin{proof}
The maximal minimal strong shift-invariant radius among gates in $S \cup \{c^0, c^2\}$ is $1$ (for $c^2$), so the expressions are well-defined and we have $S \cup \{c^0_n, c^2_n\} \subset G_n$ by Lemma~\ref{lem:Sufficientn}.

It is known that every even permutation on $\{0,1\}^n$ can be written as a finite composition of applications of the NOT gate and Toffoli gate when we are allowed to apply these gates to arbitrary $3$-tuples of wires (this is the folklore part, see e.g. \cite{BoKaSa16,Xu15}). The swaps in $S$ allow arbitrary permutations of the coordinates, so we can conjugate any $3$-tuple of coordinates to the support $\{0, 1, 2\}$ of the Toffoli gate $c^2_n$, and similarly we can apply the NOT gate in any coordinate. 

For the second claim, observe that the gates in $S$ are conjugate to $s_n$ by powers of $\sigma_n$.
\end{proof}

\begin{lemma}
\label{lem:Classical}
\[ G = \langle c^2, s, c^0, \sigma \rangle = \langle c^2, c^1, R c^1 R, c^0, \sigma \rangle \]
\end{lemma}

\begin{proof}
It is enough to show that the local permutations are generated, and by conjugating with powers of $\sigma$ (which is in both claimed generating sets), it is enough to show that those with strong support contained in $[0,n-1]$ are generated, for arbitrarily large $n$.


By the previous lemma, the leftmost generating set can perform every even permutation of the word on the interval $[0,n-1]$ for $n \geq 4$: none of the generators of $G_n$ in the first generating set in the Lemma~\ref{lem:GeneratingGn} send information over the borders of the interval $[0,n-1]$, so the action of the subgroup generated by
\[ \{\sigma^{-m} \circ s \circ \sigma^m \;|\; m \in [0, n-2]\} \cup \{ c^0, c^2 \} \]
precisely simulates the action of $G_n$ of the corresponding generators in the lemma. 

As in Lemma~\ref{lem:Sufficientn}, we see that every permutation of $\{0,1\}^n$ is even when seen as a permutation of $\{0,1\}^{n+1}$ where the rightmost bit is never changed or looked at (i.e. it is a \emph{borrowed bit} in the terminology of \cite{Xu15}), so the leftmost generating set indeed generates all of $G$.

The second equality follows from \eqref{eq:cliche}.
\end{proof}

In logical gates contexts where wire swaps are available for free, a commonly used gate set is $\{c^0, c^1, c^2\}$, known as the NCT library (NOT, Controlled NOT, Toffoli gate), and corresponds to the rightmost generating set $c^2, c^1, R c^1 R, c^0, \sigma$ above (though we need two versions of $c^1$). The gate $c^1$ is usually included as it is cheaper than $c^2$ in physical implementations, but for the purpose of theory it can simply be expressed in terms of $c^0$ and $c^2$, and thus need not be included when $s$ is included, giving the leftmost generating set $c^2, s, c^0, \sigma$.


When the wires are organized into a finite cycle, any universal set of gates for $G$ gives one for the alternating group $G_n$.

\begin{theorem}
\label{thm:RingCase}
Let $F \subset G$ be any finite generating set. Then for large enough $n$, $\{f_n \;|\; f \in F\}$ generates $G_n$.
\end{theorem}

\begin{proof}
It is enough to show that $\{f_n \;|\; f \in F\}$ generates some generating set of $G_n$. By Lemma~\ref{lem:GeneratingGn}, it is enough to generate all of $S$ and the gates $c^0_n, c^2_n$. Thus it is enough to generate the standard generating set $\{s_n, c^0_n, c^2_n, \sigma_n\}$.

The gate $s$ can be written as a finite composition of elements in $F$, i.e. $s = f^1 \circ \cdots \circ f^k$ for some $f^1,...,f^k \in F$. By Lemma~\ref{lem:WeakLEF}, for $n$ large enough we then have $s_n = f^1_n \circ \cdots \circ f^k_n$. The same argument applies to the gates in $\{c^0, c^2, \sigma\}$, so for large enough $n$ we indeed have
\[\{s_n, c^0_n, c^2_n, \sigma_n\} \subset \langle f_n \;|\; f \in F \rangle, \]
concluding the proof.
\end{proof}



The crucial point in the proof was that $G_n$ has essentially a single finite generating set that does not depend on $n$, so it is enough to write the formulas generating this generating set, and these formulas then work for large $n$.

We do not know whether the converse of Theorem~\ref{thm:RingCase}, e.g., whether the fact $\{f_n \;|\; f \in F\}$ and the shift together generate $G_n$ for all large enough $n$ implies that $F$ and $\sigma$ generate $G$. It could be the case that some set $F$ is only able to generate the standard generating set for all $n$ by sending information around the cycle. 





\section{Generating $G$ by word-swapping}
\label{sec:GeneratingByWS}

For words $u, v \in \{0, 1\}^n$, write $f_{u,v}$ for the map $f_{u,v}(x)_{[0,n-1]} = v$ if $x_{[0,n-1}] = u$, $f_{u,v}(x)_{[0,n-1]} = u$ if $x_{[0,n-1}] = v$, $f_{u,v}(x)_{[0,n-1]} = x_{[0,n-1]}$ otherwise, and $f_{u,v}(x)_i = x_i$ for $i \notin [0,n-1]$. For example the Toffoli gate is $c^2 = f_{011,111}$.

We give a classification of the pairs of words $u, v$ such that $f_{u, v}$ is a universal gate together with the shift and the flip, i.e. $\langle f_{u, v}, \sigma, c^0 \rangle = G$. The characterization is that the words differ in exactly one place, which is not at the border.

\begin{theorem}
\label{thm:TwoWords}
Let $u', v' \in \{0,1\}^n$ be distinct words and let $G' = \langle c^0, f_{u',v'}, \sigma \rangle$. Then $G = G'$ if and only if $D(u', v') \in 0^*0100^*$.
\end{theorem}


\begin{proof}
For sufficiency, suppose the RHS holds. Since $c^0, \sigma \in G'$, we can conjugate $u'$ to $u = 0^n$ and $v'$ to $v = D(u', v')$. By applying the gate $f_{0^n, v}$, flipping the $i$th bit on the tape, applying the gate again, and again flipping the $i$th bit, we effectively eliminate one of the bits that $f_{0^n, v}$ looks at. In symbols,
\[ \sigma c^0 f_{0^n, 0v} c^0 f_{0^n, 0v} \sigma^{-1} = f_{0^{n-1},v} \]
and
\[ \sigma^{-n} c^0 \sigma^n f_{0^n, v0} \sigma^{-n} c^0 \sigma^n f_{0^n, v0} = f_{0^{n-1},v}. \]

By the assumption, $v$ contains the subwords $01$ and $10$, so $c^1$ and $Rc^1R$ are in $G'$ by applying the formulas of the previous paragraph, eliminating all but the subword $01$ (resp. $10$) of $v$. Then by \eqref{eq:cliche}, also $s \in G'$. 

Now, use the same trick to eliminate all but the subword $010$ from $v$ to obtain $f_{000,010} \in G'$. This can be conjugated to the Toffoli gate $c^2$ by applying suitable shifts of the swap and the flip. It then follows from Lemma~\ref{lem:Classical} that $G' = G$.

As for necessity, if $D(u', v') \in 0^*$, then $f_{u', v'}$ is the trivial gate. If $D(u', v') \in 0^*1$, then $\{\sigma, c^0, f_{u', v'}\} \subset G_R$. If $D(u', v') \in 10^*$, then $\{\sigma, c^0, f_{u', v'}\} \subset G_L$.

If $D(u', v')$ contains more than one $1$, then consider $w = {^\omega 0}.D(u', v')0^\omega$ as a vector in $X_0$. Together with its shifts, it generates a shift-invariant vector space $V \leq X_0$. It does not generate all of $X_0$: one way to see this is that the support of a sum of distinct shifts of $w$ has cardinality at least $2$, since the leftmost coordinate of the leftmost summand and the rightmost coordinate of the rightmost summand are not cancelled. It is straightforward to show that $\{c^0, f_{u',v'}, \sigma\} \subset G_V$.
\end{proof}

In particular the previous theorem shows that $c^0, c^2, \sigma$ do not generate $G$, since the Toffoli gate $f_{011, 111}$ satisfies $D(011, 111) = 100 \notin 0^*0100^*$ (and the reason for this is that $c^2 \in G_L$).

One cannot omit $c^0$ from the theorem, since we can never have $G = \langle f_{u,v}, \sigma \rangle$ for the simple reason that either $0^\Z$ or $1^\Z$ is a fixed point for both generators when $D(u, v) \in 0^*0100^*$, and when $D(u, v) \notin 0^*0100^*$, even $\{f_{u,v}, c^0, \sigma\}$ does not generate $G$ by the theorem.

\section{Asynchronous application of elementary CA}

We can now find a single inert element that generates $G$ together with the shift, i.e a \emph{singleton universal gate set}. This is an optimal way to generate $G$ in several ways: we need only one gate, which only modifies one cell at a time, and which has strong radius $r = 1$. This radius is optimal since all permutations of two or fewer bits are affine.

We recall the \emph{Wolfram number} \cite{Wo83} of an elementary cellular automaton. Enumerate words $\{0,1\}^3$ in reverse lexicographic order as $w_0 = 111, w_1 = 110, w_2 = 101, ...$. To each $b \in \{0,1\}^8$, we associate a self-homeomorphism $e^b$ of $X$ defined by $e^b(x)_0 = b_j$ where $w_j = x_{[-1,1]}$, and $e^b(x)_i = x_i$ for $i \neq 0$. This number $b$ is then written in decimal.
We show that the singleton universal gate sets among $e^n$ are precisely $e^{57} = e^{00111001_2}$ and $e^{99} = e^{01100011_2}$.

Usually, a cellular automaton applies the local rule to all cells of $\Z$ at once. The definition here is the definition of (fully) asynchronous application of the cellular automaton, as studied in \cite{MaMcMo11,Vi12}.

The map $e^{57}$ can be described as follows: $e^{57}(x)_0 = 1-x_0$ unless $x_{-1} = 0$ and $x_1 = 1$, and no other cell is changed. Thus, $e^{57} = c^0 \circ \sigma \circ f_{001, 011} \circ \sigma^{-1}$. The map $e^{99}$ is its mirror image.

\begin{theorem}
\label{thm:Main}
The function $e^n$ is a universal gate in $G$ if and only if $n \in \{57, 99\}$.
\end{theorem}

In particular, $G = \langle e^{57}, \sigma \rangle = \langle e^{99}, \sigma \rangle$ is $2$-generated.

\begin{proof}
Consider general $e^n, n \in \{0,1,...,255\}$. For $e^n$ to be in $G$, the binary representation of $n$ must be of the form
\[ n = (ab(1-a)(1-b)cd(1-c)(1-d))_2. \]
If $d = 1$, then $0^\Z$ is fixed, and similarly we must have a = 0, so for universality we must have
$n = (0 b 1 (1-b) c 0 (1-c) 1)_2$,
so the choices are among $00110011, 00111001, 01100011, 01101001$. These are $e^{51}, e^{57}, e^{99}, e^{105}$ respectively. We have $e^{51} = c^0$, so it is not a singleton universal gate set by Lemma~\ref{lem:Subgroups}. The gate $e^{105}$ sums the bits at $x_{-1}, x_1$ to the bit at $x_0$, and then flips the value at $x_0$. This gate is thus affine, thus cannot be universal. We are left with only the choices $e^{57}$ and $e^{99}$.

Let us show that $e^{57}$ is universal, i.e. $\langle e_{57}, \sigma \rangle = G$.
By Theorem~\ref{thm:TwoWords}, it is enough to show that $e^{57}$ and $\sigma$ generate $c^0$, as then
\[ e^{57} = c^0 \circ \sigma \circ f_{001, 011} \circ \sigma^{-1} \implies
f_{001, 011} = \sigma^{-1} c^0 e^{57} \sigma
\] 
and Theorem~\ref{thm:TwoWords} shows that $G$ is generated by $\{c^0, f_{001, 011}, \sigma\}$.

Writing $a = \sigma \circ e^{57} \circ \sigma^{-1}$, $b = e^{57}$, $c = \sigma^{-1} \circ e^{57} \circ \sigma$ and omitting $\circ$ from the notation, it is easy to verify (for example by computer) that
\[ c^0 = abcabcbababacbabababcbcabacbabcbcbcbcabcbcbabacbcb. \]
Python code for checking the above identity is included, see Appendix~\ref{sec:Code}.

We have $e^{99} = R e^{57} R$, so $e^{99}$ is also universal since $R \langle \sigma \rangle R = \langle \sigma \rangle$ and $RGR = G$).
\end{proof}

Translating the proof to concrete expressions for the standard gates leads to rather long compositions of shifts of $e^{57}$. A straight-line (context-free) grammar for generating such expression is given in Appendix~\ref{sec:SLG}, and we provide the strings themselves in Appendix~\ref{sec:Strings}.

Note that though one gate can be universal, one element of $G$ cannot generate it, since $G$ is not abelian, thus not cyclic. 

The minimal number $3$ of shift-conjugates that needed to be considered was checked by GAP, and the expression itself was found by a naive Python search, by computing the ball of radius $25$ and finding $g, h$ such that $gh$ is the flip (according to my program, 49 generators do not suffice).

By Theorem~\ref{thm:RingCase}, this also solves Conjecture 5.10 in \cite{MaMcMo11}, also asked in \cite{Vi12}.

\begin{corollary}
For all $n \geq 4$, $G_n = \langle e^{57}_n, \sigma_n \rangle = \langle e^{99}_n, \sigma_n \rangle$.
\end{corollary}

\begin{proof}
This is known for ring sizes up to $8$ \cite{MaMcMo11} (and checked up to 10 in \cite{Vi12}). Theorem~\ref{thm:RingCase} (more precisely its proof) kicks in at ring size $5$.
\end{proof}

As stated, the concrete implementations we give in the appendix work starting only from $n = 8$ -- but actually the formulas happen to work already for $n = 4$.

\section{Questions and future work}

We do not know what the general conditions are under which a single gate, together with the shift, is universal, but this is semidecidable by Lemma~\ref{lem:Classical} and Theorem~\ref{thm:Main}. More generally, the set of finite sets of gates which are universal for $G$ is a $\Sigma^0_1$ set. We do not know whether non-universality is semidecidable, and do not know whether the problem is $\Sigma^0_1$-complete.

Given a finite set of gates $F$, it is not clear whether it is semidecidable in either direction whether $F_n$ together with the shift $\sigma_n$ generates $G_n$ for all large enough $n$. Without some additional argument, universal gate sets only form a $\Sigma^0_2$ set in this context.

Note also that $e_{57}$ is the composition of a word swap and the flip. We do not have a characterization of word swaps whose composition with the flip is a singleton universal gate set. A related question is whether a random inert gate in $\Sym(\{0,1\}^n)$, together with the shift, generates $G$ with high probability (as $n \rightarrow \infty$).

We asked a related question also in \cite{Sa18a}, namely whether such a gate is even (almost always) universal as a gate on $[0,n-1]$, when we are not allowed to apply it over the borders. We do not know any gates that are universal in this stronger sense; that gate $e^{57}$ is not universal in this sense.

We do not know whether $G$ is finitely presented. A presentation for the groupoid of reversible gates is also not known \cite{Se16}.

In \cite{AaGrSc15} the analog of Post's lattice for all reversible bit operations is fully described. Since they allow ancilla bits (bits containing some prescribed value, and which need not return to their initial value in the end), it is not immediately clear that this classification implies something in our framework. A full description of the lattice of subgroups of $G$ containing the shift and the swap would amount to classifying the lattice of reversible transformations implementable using \emph{borrowed bits}, in the terminology of \cite{Xu15}; we do not know if there is a known classification. We have not checked whether the variated Toffoli gate of \cite{Xu15} is a universal singleton gate set in our sense.

We note that, unlike the lattice of \cite{AaGrSc15} and Post's lattice, the lattice of subgroups of $G$ containing the shift (but not necessarily the swap) is uncountable -- it is easy to embed, for example, the infinite direct sum $\bigoplus_{n \in \N} \Z_2$, which has uncountably many subgroups.

There are obvious generalizations of $G$ which seem interesting. It is a standard direction of generalization in the theory of reversible gates to change the binary alphabet to a higher-arity one, and sometimes the qualitative properties change, e.g. the parity of the finitely-generated part of the group of reversible gates depends on the parity of the alphabet \cite{Bo15,Se16,BoKaSa16}. Here, the fact we have a geometry for the arrangement of the wires allows us to do much more: we can change the geometry to an arbitrary group, and instead of just increasing the size of the alphabet, we can replace the set of legal configurations by a subshift.

Finite generating sets for the analog of $G$ on any full shift follow from results about reversible gates (see \cite{BoKaSa16} for the general case) as in Lemma~\ref{lem:Classical}. We leave open whether a universal gate exists on every alphabet, and more generally whether one exists on every mixing SFT \cite{LiMa95}, when using the obvious generalizations.

If $\Z$ is replaced by a group $\mathcal{G}$, then the corresponding $G = G_{\mathcal{G}}$ is a (locally finite)-by-$\mathcal{G}$ group, and thus inherits many properties of $\mathcal{G}$. It is finitely generated if and only if $\mathcal{G}$ is, essentially by Lemma~\ref{lem:Classical}. If $\mathcal{G} = \Z^d$, then since $\Z^d$ is an abelian epimorphic image of rank $d$ of $G_{\Z^d}$, at least $d$ generators are needed. Is $G_{\Z^d}$ generated by $d$ elements, or at least $d+1$ elements, for $d > 1$? Is there always a universal inert gate? Again, one can also ask such questions for larger alphabets and other subshifts.

\section*{Acknowledgements}

We thank Ilkka T\"orm\"a for useful suggestions. 

\bibliographystyle{plain}
\bibliography{../../../bib/bib}{}

\newpage

\appendix

\section{Straight-line grammar for the standard generators}
\label{sec:SLG}

We give a context-free grammar, obtained directly from the proof, that describes the standard generators for the swap $s$, the flip $c^0$, the CNOTs $c^1, Rc^1R$, and the Toffoli gate $c^2$, more precisely the grammar generates a single string over the alphabet $\{1,2,...,6\}$ whose letters indicate the cell of $[0,n]$ ($n \geq 7$) where the $e^{57}$ gate should be applied. The grammar generates the standard gates in cell $4 \in [0,n]$ cell when the start symbol is varied.

In the grammar we use an extension of the NCT naming scheme: $N_i$ applies $c^0$ at $i$ (i.e. it applies $\sigma^{-i} \circ c^0 \circ \sigma^i$), $C_i$ corresponds to $\sigma^{-i} \circ c^1 \circ \sigma^i$ and $T_i$ corresponds to $\sigma^{-i} \circ c^2 \circ \sigma^i$. The swap of cells $i$ and $i+1$ is $S_i$, and $\sigma^{-i} \circ R \circ c^1 \circ R \circ \sigma^i$ is $D_i$. The terminal symbol $i \in \{1,...,6\}$ is interpreted as $\sigma^{-i} \circ e^{57} \circ \sigma^i$.

The grammar is $G = (V, \Sigma, R, S')$, where the start symbol is the element of $\{N_3, C_3, T_3, D_3, S_3\}$ we want to generate,
\[ V = \{N_3, C_3, T_3, D_3, S_3\} \cup \{ N_2, N_4, N_5, D_4, \bar E_3, \bar E_4 \}, \]
\[ \Sigma = \{1,2,3,4,5,6\}, \]
and $R$ contains the following rewrite rules:
\[ T_3 \rightarrow S_3 N_3 \bar E_4 N_3 S_3 \]
\[ S_3 \rightarrow C_3 D_4 C_3 \]
\[ C_3 \rightarrow \bar E_3 N_2 \bar E_3 N_2 \]
\[ D_3 \rightarrow N_2 \bar E_3 N_4 \bar E_3 N_2 N_4 \]
\[ D_4 \rightarrow N_3 \bar E_4 N_5 \bar E_4 N_3 N_5 \]
\[ \bar E_3 \rightarrow N_3 3 \]
\[ \bar E_4 \rightarrow N_4 4 \]
\[ N_2 \rightarrow 12312321212132121212323121321232323231232321213232 \]
\[ N_3 \rightarrow 23423432323243232323434232432343434342343432324343 \]
\[ N_4 \rightarrow 34534543434354343434545343543454545453454543435454 \]
\[ N_5 \rightarrow 45645654545465454545656454654565656564565654546565 \]

\newpage 

\section{The standard gates}
\label{sec:Strings}

We list the strings generated by the straight-line grammar from Section~\ref{sec:SLG} from start symbols $N_3, C_3, T_3, D_3, S_3$, i.e. implementations of standard gates as shifts of  $e^{57}$. 

$N_3 \rightarrow $
\printnumber{23423432323243232323434232432343434342343432324343}

$C_3 \rightarrow $
\printnumber{1231232121213212121232312132123232323123232121323232342343232324323232343423243234343434234343232434312312321212132121212323121321232323231232321213232323423432323243232323434232432343434342343432324343}

$T_3 \rightarrow $
\printnumber{123123212121321212123231213212323232312323212132323234234323232432323234342324323434343423434323243431231232121213212121232312132123232323123232121323232342343232324323232343423243234343434234343232434345645654545465454545656454654565656564565654546565234234323232432323234342324323434343423434323243434345345434343543434345453435434545454534545434354544564565454546545454565645465456565656456565454656543453454343435434343454534354345454545345454343545423423432323243232323434232432343434342343432324343123123212121321212123231213212323232312323212132323234234323232432323234342324323434343423434323243431231232121213212121232312132123232323123232121323232342343232324323232343423243234343434234343232434323423432323243232323434232432343434342343432324343434534543434354343434545343543454545453454543435454234234323232432323234342324323434343423434323243431231232121213212121232312132123232323123232121323232342343232324323232343423243234343434234343232434312312321212132121212323121321232323231232321213232323423432323243232323434232432343434342343432324343456456545454654545456564546545656565645656545465652342343232324323232343423243234343434234343232434343453454343435434343454534354345454545345454343545445645654545465454545656454654565656564565654546565434534543434354343434545343543454545453454543435454234234323232432323234342324323434343423434323243431231232121213212121232312132123232323123232121323232342343232324323232343423243234343434234343232434312312321212132121212323121321232323231232321213232323423432323243232323434232432343434342343432324343}

$D_3 \rightarrow $
\printnumber{34534543434354343434545343543454545453454543435454123123212121321212123231213212323232312323212132323234234323232432323234342324323434343423434323243433453454343435434343454534354345454545345454343545432342343232324323232343423243234343434234343232434312312321212132121212323121321232323231232321213232}

$S_3 \rightarrow $
\printnumber{1231232121213212121232312132123232323123232121323232342343232324323232343423243234343434234343232434312312321212132121212323121321232323231232321213232323423432323243232323434232432343434342343432324343456456545454654545456564546545656565645656545465652342343232324323232343423243234343434234343232434343453454343435434343454534354345454545345454343545445645654545465454545656454654565656564565654546565434534543434354343434545343543454545453454543435454234234323232432323234342324323434343423434323243431231232121213212121232312132123232323123232121323232342343232324323232343423243234343434234343232434312312321212132121212323121321232323231232321213232323423432323243232323434232432343434342343432324343}

\newpage

\section{Checking the identity in Theorem~\ref{thm:Main}}
\label{sec:Code}

\begin{lstlisting}[language=Python]
def a(l):
    if l[0] == 0 and l[2] == 1:
        return l
    return [l[0], 1-l[1], l[2], l[3], l[4]]

def b(l):
    if l[1] == 0 and l[3] == 1:
        return l
    return [l[0], l[1], 1-l[2], l[3], l[4]]

def c(l):
    if l[2] == 0 and l[4] == 1:
        return l
    return [l[0], l[1], l[2], 1-l[3], l[4]]

def center_flipped(l, s):
    orl = l
    for i in s:
        if i == "a":
            l = a(l)
        if i == "b":
            l = b(l)
        if i == "c":
            l = c(l)
    center_is_flipped = orl[2] != l[2]
    others_are_not = orl[:2] + orl[3:] == l[:2] + l[3:]
    return center_is_flipped and others_are_not

def binw(n):
    if n == 0:
        yield []
    else:
        for i in [0,1]:
            for b in binw(n-1):
                yield [i] + b

# We actually apply from left to right here, so
# s should be reversed, but since the generators are
# involutions, this does not make a difference.
s = "abcabcbababacbabababcbcabacbabcbcbcbcabcbcbabacbcb"
                
for l in binw(5):
    if not center_flipped(l, s):
        print ("It is not the flip!")
        break
else:
    print ("It is the flip.")
\end{lstlisting}

\end{document}